\documentclass[11pt]{article}
\usepackage{color}
\usepackage{amsthm,amssymb,amsmath,latexsym}
\usepackage[all]{xy}			% commutative diagrams
\usepackage{hyperref,url}
\usepackage{graphicx}

\graphicspath{{figures/}}

\newtheorem{theorem}{Theorem}[section]
\newtheorem{corollary}[theorem]{Corollary}
\newtheorem{lemma}[theorem]{Lemma}
\newtheorem{proposition}[theorem]{Proposition}

\theoremstyle{definition}
\newtheorem{definition}[theorem]{Definition}
\newtheorem{example}[theorem]{Example}
\newtheorem{remark}[theorem]{Remark}

\newcommand{\rep}{\textnormal{rep}}
\newcommand{\val}{\textnormal{val}}
\newcommand{\F}{\mathbb F}
\newcommand{\N}{\mathbb N}
\newcommand{\Q}{\mathbb Q}
\newcommand{\Z}{\mathbb Z}

\newcommand{\nequiv}{\mathrel{\not\equiv}}
\newcommand{\colonequal}{\mathrel{\mathop:}=}

\begin{document}

\title{Profinite automata}

\date{November 12, 2016}

\author{{\bf Eric Rowland}\footnote{Supported by a Marie Curie Actions COFUND fellowship.}\\
University of Liege, Belgium\\
%Current address:
Hofstra University, Hempstead, NY, USA \\
\\
{\bf Reem Yassawi}\\
Trent University, Peterborough, Canada\\
IRIF,   CNRS UMR 8243, Universit\'{e}  Paris-Diderot, %Ñ Paris 7,%Case 7014, 75205 Paris Cedex 13, 
France}
\maketitle

\begin{abstract}
Many sequences of $p$-adic integers project modulo $p^\alpha$ to $p$-automatic sequences for every $\alpha \geq 0$.
Examples include algebraic sequences of integers, which satisfy this property for every prime $p$, and some cocycle sequences, which  we show satisfy this property for a fixed $p$.
For such a sequence, we construct a profinite automaton that projects modulo $p^\alpha$ to the automaton generating the projected sequence.
In general, the profinite automaton has infinitely many states.
Additionally, we consider the closure of the orbit, under the shift map, of the $p$-adic integer sequence, defining a shift dynamical system.
We describe how this shift is a letter-to-letter coding of a  shift generated by a constant-length substitution  defined on an uncountable alphabet, and we establish some dynamical properties of these shifts.
\end{abstract}

\section{Introduction}\label{Introduction}

A {\em substitution} (or {\em morphism}) on an alphabet $\mathcal A$ is a map $\theta: \mathcal A \rightarrow \mathcal A^*$, extended to $\mathcal A^\N$ by concatenation.
A substitution is {\em length-$k$} (or \emph{$k$-uniform}) if, for each $a \in \mathcal A$, the length of $\theta(a)$ is $k$. The extensive literature on substitutions has traditionally focused on the case where $\mathcal A$ is finite. Some exceptions include recent work, for example in \cite{ferenczi} and \cite{mauduit}.
Substitutions on a countably infinite alphabet have been used to describe lexicographically least sequences on $\N$ avoiding certain patterns~\cite{Guay-Paquet--Shallit, Rowland--Shallit}, and they have been used in the combinatorics literature to enumerate permutations avoiding patterns~\cite{west}.

In this article we present a new construction for constant-length substitutions on an uncountable alphabet.
Our motivation comes from the following classical results.
Let  $( a(n))_{n \geq 0}$ be an {\em automatic sequence} (see Definition~\ref{dfao}). Cobham's theorem (Theorem~\ref{Cobham})  characterizes an automatic sequence as the coding, under a letter-to-letter map, of a fixed point of a constant-length substitution. Christol's theorem (Theorem~\ref{Christol}) characterizes $p$-automatic sequences for prime $p$; they are precisely the sequences whose generating function is algebraic over a finite field of characteristic $p$. The following can be viewed as a generalization of one direction of Christol's characterization.

\begin{theorem}[{\cite[Theorem 32]{Christol},  \cite[Thereom 3.1]{Denef--Lipshitz}}]\label{Denef--Lipshitz_1}
Let  $(a(n))_{n\geq 0}$ be a sequence of $p$-adic integers such that $\sum_{n \geq 0} a(n) x^{n}$ is algebraic over $\Z_p(x)$, and let $\alpha \geq 0$.
Then $(a(n) \bmod p^\alpha)_{n \geq 0}$ is $p$-automatic. 
\end{theorem}

Thus certain  $p$-adic integer sequences (and, in particular, integer sequences) have the property that they become $p$-automatic when reduced modulo $p^\alpha$, for every $\alpha \geq 0$.
More generally, the diagonal of a multivariate rational power series is $p$-automatic when reduced modulo $p^\alpha$, and one can explicitly compute an automaton for $(a(n) \bmod p^\alpha)_{n \geq 0}$ for all but finitely many primes $p$~\cite[Theorem 2.1]{Rowland--Yassawi}.

Fix a prime $p$, and let $(a(n))_{n\geq 0}$ be a sequence such that $(a(n) \bmod p^\alpha)_{n \geq 0}$ is $p$-automatic for every $\alpha \geq 0$.
For each $\alpha$, there is a finite automaton generating $(a(n) \bmod p^\alpha)_{n \geq 0}$.
In Lemma~\ref{inverse limit of automata} we show that these automata can be chosen in a compatible way; namely, their inverse limit exists.
In this way we obtain a \emph{profinite automaton}  (Definition~\ref{profinite_automaton}) generating the sequence $(a(n))_{n\geq 0}$.

We can obtain other inverse limit objects from a $p$-adic integer sequence in a similar way.
In particular, Cobham's theorem guarantees a length-$p$ substitution $\theta_\alpha$ such that $(a(n) \bmod p^\alpha)_{n \geq 0}$ is a coding of a fixed point of $\theta_\alpha$.
Each substitution $\theta_\alpha$ is a substitution on a finite alphabet, but their inverse limit is a profinite substitution on an alphabet that is, in general, uncountable (Theorem~\ref{substitution_theorem}).
This alphabet has a natural coding to the set $\Z_p$ of $p$-adic integers, and the   sequence $(a(n))_{n\geq 0}$ is the coding of a fixed point of the profinite substitution.

With this profinite substitution, we obtain a shift (Theorem~\ref{commuting_diagram_lemma_2}), as in the classical finite-alphabet case. This shift is the closure of the orbit, under the shift map, of a fixed point (or coding of a fixed point) of the profinite substitution.
One feature of profinite substitutions is that their shifts live in a compact topological space, and many of the classical results on primitive substitution shifts carry through to our setting. For example, if we assume that each substitution $\theta_\alpha$ is primitive, then  the profinite shift is both minimal (Proposition~\ref{minimal}) and uniquely ergodic (Corollary~\ref{choksi_corollary}). With the same assumption, the maximal equicontinuous factor of the profinite shift is an odometer, and any measurable eigenvalue is continuous (Theorem~\ref{discrete_spectrum}). Finally, profinite substitutions are recognizable (Theorem~\ref{recognizable}).
However it is not clear when  a $p$-adic integer sequence generates primitive substitutions  $\theta_\alpha$. In one example we consider,  not only are  the substitutions $\theta_1$ and $\theta_2$ not primitive,  but their shifts contain shift-periodic sequences (Example~\ref{Catalan modulo 2}).

In Section~\ref{limit sets}, we discuss how profinite substitutions allow us to view certain $p$-adic limits in terms of attractor sets for dynamical systems.
For example, consider the Fibonacci sequence $F(n)_{n\geq 0} = 0, 1, 1, 2, 3, 5, 8, 13, \dots$.
The following graphic shows the hundred least significant binary digits of $F(2^n)$ for each $0 \leq n \leq 20$, where $0$ is represented by a white cell, $1$ is represented by a black cell, and digits increase in significance to the left.
\begin{center}
	\scalebox{.5}{\includegraphics{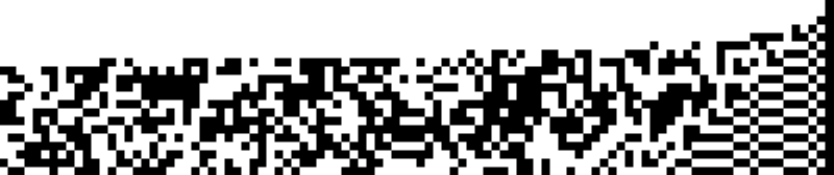}}
\end{center}
As $2$-adic integers, the two limit points of the sequence $F(2^n)_{n\geq 0}$ are $\pm \sqrt{-\frac{3}{5}}$ in $\Z_2$~\cite{Rowland--Yassawi_constant-recursive}.
Similar behavior is seen for the sequence $C(n)_{n \geq 0} = 1, 1, 2, 5, 14, 42, 132, 429, \dots$ of Catalan numbers, where $C(n) = \frac{1}{n+1} \binom{2n}{n}$.
The following shows the binary digits of $C(2^n)$ for $0 \leq n \leq 20$.
\begin{center}
	\scalebox{.5}{\includegraphics{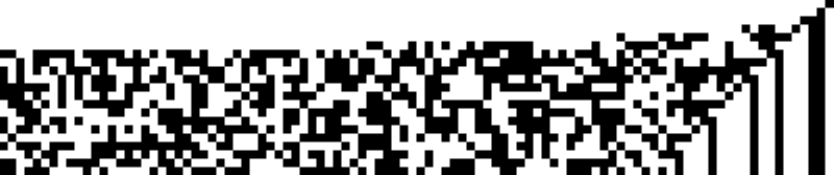}}
\end{center}
As this graphic suggests, the sequence $C(2^n)_{n \geq 0}$ converges in $\Z_2$~\cite{michel_miller_rennie}.
These limits are elements in the limit sets of certain profinite shifts. In particular, such a limit is a component of an inverse limit of $\theta_\alpha$-periodic points.

The outline of the article is as follows.
In Section~\ref{Notation} we recall the major facts about automatic sequences that we shall use.
In Section~\ref{generated by automatic sequences} we construct profinite automata as inverse limits of finite automata. In Section~\ref{dynamical properties} we define profinite substitutions and their shifts, and we establish various dynamical properties  of the latter.   
In Section~\ref{Cocycle sequences} we study a family of sequences  that arise in a dynamical context (from cocycle maps),
 which are not algebraic, but which for a single prime $p$ generate an inverse limit substitution dynamical system (Theorem~\ref{cocycle}).

\section{Background on automatic sequences}\label{Notation}

In this section we establish notation and give the necessary properties of automatic sequences.

\subsection{Finite automata and automatic sequences}

We first give the formal definition of an automaton.

\begin{definition}\label{dao}
Let $k \geq 2$.
A {\em $k$-deterministic  automaton with output} ($k$-DAO) is a 6-tuple  $(\mathcal S, \Sigma_{k}, \delta, \overline{s}, \mathcal A, \tau)$, where $\mathcal S$ is a set of ``states'', $\overline{s} \in \mathcal S$ is the {\em initial state}, $\Sigma_k=\{0,1, \ldots, k-1\}$,
$\mathcal A$ is an alphabet,
$\tau:\mathcal S\rightarrow \mathcal A$ is the \emph{output function}, and $\delta:\mathcal S\times \Sigma_{k}\rightarrow \mathcal S$ is the {\em transition function}.
\end{definition}

In this section we shall be concerned with \emph{finite} automata.

\begin{definition}\label{dfao}
Let $k \geq 2$.
A {\em $k$-deterministic finite automaton with output} ($k$-DFAO) is a $k$-DAO whose set $\mathcal S$ of states is finite.
\end{definition}

The function $\delta$ extends to the domain $\mathcal S \times \Sigma_k^*$ by defining $\delta(s, \epsilon) \colonequal s$ for the empty word $\epsilon$ and recursively defining $\delta(s, n_\ell n_{\ell-1} \cdots n_0) \colonequal \delta(\delta(s, n_\ell), n_{\ell-1} \cdots n_0)$.
Given a natural number $n$ and an integer $k \geq 2$, we write $\rep_k(n) = n_\ell  \cdots n_1 n_0\in \Sigma_k^*$ for the standard base-$k$ representation of $n$ where $n = n_\ell k^\ell + \dots + n_1 k + n_0$  and $n_\ell\neq 0$.
We can feed $\rep_k(n)$, beginning with the most significant digit $n_\ell$, into an automaton as follows. 
(Recall that the standard base-$k$ representation of $0$ is the empty word.)

\begin{definition}\label{automatic sequence}
A sequence $( a(n))_{n \geq 0}$ of elements in $\mathcal A$ is {\em $k$-automatic} if there is a $k$-DFAO  
$\mathcal M = (\mathcal S,\Sigma_{k},\delta, \overline{s}, \mathcal A, \tau)$ such that $a(n) = \tau(\delta(\overline{s}, \rep_k(n)))$ for each $n \geq 0$.
\end{definition}

For ease of notation we will write  $ a(n)_{n \geq 0}$ for $( a(n))_{n \geq 0}$.
When we obtain an automatic sequence  $a(n)_{n \geq 0}$ as in Definition~\ref{automatic sequence} by reading the most significant digit first, we say that the automaton  generates $a(n)_{n \geq 0}$ in {\em direct reading}.
In this article our automata will always generate sequences in direct reading. (This is in contrast to the automata in \cite{Rowland--Yassawi}, where our algorithms produced machines that generated the desired sequence by reading the least significant digit first.)

\begin{example}\label{Catalan modulo 4}
Consider the following automaton for $k = 2$.
Each of the six states is represented by a vertex, labelled with its output under $\tau$.
Edges between vertices illustrate $\delta$.
The unlabelled edge points to the initial state.
\begin{center}
	\includegraphics[scale=.8]{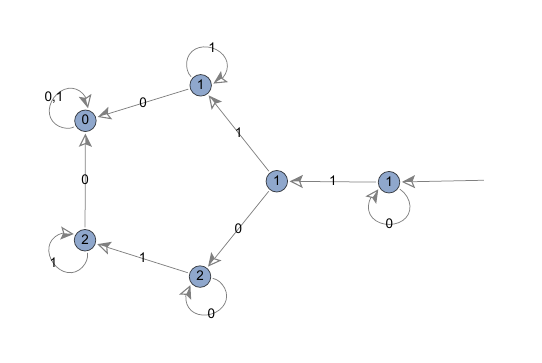}
\end{center}
The $2$-automatic sequence produced by this automaton in direct reading is
\[
	a(n)_{n \geq 0} = 1, 1, 2, 1, 2, 2, 0, 1, 2, 2, 0, 2, 0, 0, 0, 1, \dots.
\]
\end{example}

\subsection{Some characterizations of automaticity}

\begin{definition}\label{substitution}
Let $\mathcal A$ be an alphabet, which we do not assume to be finite. A {\em substitution}  is a map $\theta: {\mathcal A} \rightarrow {\mathcal
A}^*$.
The map  $\theta$ extends  to maps $\theta: {\mathcal
A}^* \rightarrow {\mathcal A}^*$ and $\theta: {\mathcal A}^\N \rightarrow {\mathcal A}^\N$ by concatenation:
$\theta(a(n)_{n\geq 0}) \colonequal \theta(a(0))
\cdots \theta(a(k))\cdots$. 
If there is some $k$ with $|\theta(a)|=k$ for each $a \in \mathcal A$, then we say that $\theta$ is a {\em length-$k$} substitution.
A {\em fixed point} of $\theta$ is a sequence $ a(n)_{n \geq 0} \in {\mathcal A}^\N$ such that $\theta(  a(n)_{n \geq 0}     ) = a(n)_{n \geq 0}$.
Let $\tau: \mathcal A \rightarrow \mathcal B$ be a map; it induces a {\em letter-to-letter projection} $\tau:\mathcal A^{\N} \rightarrow \mathcal B^{\N}$. 
\end{definition}

Given a letter $s$ such that $\theta(s)$ starts with $s$, the words $\theta^n(s)$ converge, as $n\rightarrow\infty$, to a fixed point of $\theta$.
We recall Cobham's theorem~\cite{C}, which  gives us
the following characterization of automatic sequences.

\begin{theorem}\label{Cobham}
A sequence  is $k$-automatic if and only  if it is the image, under a letter-to-letter projection, of a fixed point of a length-$k$ substitution on a finite alphabet.
\end{theorem}

 If a sequence $a(n)_{n \geq 0}$ is generated in direct reading by $\mathcal M  = (\mathcal S, \Sigma_{k},\delta, \overline{s}, \mathcal A, \tau)$, then the transition map $\delta$ gives us the substitution described by Cobham's theorem.
Namely, let $\theta(s) \colonequal \delta(s,0) \delta(s,1)\cdots  \delta(s,k-1)$ for each state $s \in \mathcal S$.
If $\delta(\overline{s}, 0) \neq \overline{s}$, introduce a new letter $\overline{s}'$ and let $\theta(\overline{s}') \colonequal \overline{s}' \delta(\overline{s},1) \cdots \delta(\overline{s},k-1)$.

\begin{example}\label{substitution construction}
Let us construct a substitution $\theta$ and projection $\tau$ that generate the sequence produced by the automaton $\mathcal M$ in Example~\ref{Catalan modulo 4} in direct reading.
Removing the output function from $\mathcal M$ gives the following automaton, where we have named the states $s_0, s_1, \dots, s_5$.
\begin{center}
	\includegraphics[scale=.8]{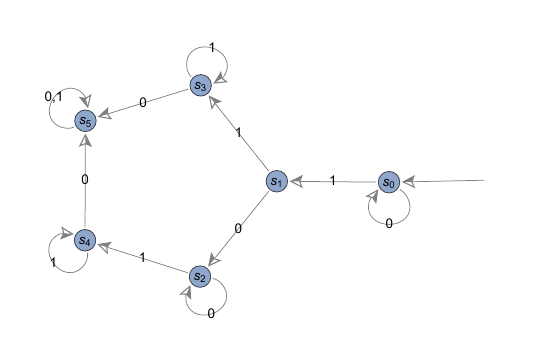}
\end{center}
This automaton dictates that $\theta$ is the length-$2$ substitution on the alphabet $\{s_0, s_1, s_2, s_3, s_4, s_5\}$ defined by
\[
	\begin{array}{lll}
		\theta(s_0) = s_0 s_1		& & \theta(s_3) = s_5 s_3 \\
		\theta(s_1) = s_2 s_3		& & \theta(s_4) = s_5 s_4 \\
		\theta(s_2) = s_2 s_4		& & \theta(s_5) = s_5 s_5.
	\end{array}
\]
The sequence
\[
	\theta^\infty(s_0) = s_0, s_1, s_2, s_3, s_2, s_4, s_5, s_3, s_2, s_4, s_5, s_4, s_5, s_5, s_5, s_3, \dots
\]
is a fixed point of $\theta$.
The letter-to-letter projection $\tau$ is the output function for $\mathcal M$, namely
\begin{align*}
	\tau(s_5) &= 0 \\
	\tau(s_0) = \tau(s_1) = \tau(s_3) &= 1 \\
	\tau(s_2) = \tau(s_4) &= 2.
\end{align*}
Therefore
\[
	\tau(\theta^\infty(s_0)) = 1, 1, 2, 1, 2, 2, 0, 1, 2, 2, 0, 2, 0, 0, 0, 1, \dots
\]
is the sequence produced by $\mathcal M$ in direct reading.
\end{example}
 
 While Theorem~\ref{Cobham}  characterizes $k$-automatic sequences for all $k \geq 2$, Christol's theorem characterizes $p$-automatic sequences for prime $p$.
By taking a sufficiently large finite field of characteristic $p$, we may assume (by choosing an arbitrary embedding) that the output alphabet is a subset of this field.

\begin{theorem}[Christol et al.~\cite{CKMR}]\label{Christol} Let 
 $a(n)_{n \geq 0}$ be a sequence of elements in $\F_{p^\alpha}$. Then
 $\sum_{n \geq 0}a (n)x^{n}$ is algebraic over $\F_{p^\alpha}(x)$ if and only if 
$a(n)_{n \geq 0}$ is $p$-automatic.
\end{theorem}

\section{Profinite automata generated by automatic sequences}\label{generated by automatic sequences}

In this section we study infinite automata that are inverse limits of finite automata.

We use the following notation for the remainder of the article.
For $0\leq \alpha\leq \beta$, let $\pi_{\alpha, \beta}: \Z/(p^{\beta}\Z) \rightarrow \Z/ (p^{\alpha}\Z)$  denote the natural projection map modulo $p^\alpha$.
The ring of $p$-adic integers is the inverse limit $\Z_p = \varprojlim \Z/(p^\alpha \Z)$ in the category of rings.
Let $\pi_{\alpha, \infty}$ be the projection map $\pi_{\alpha, \infty}: \Z_p \rightarrow \Z / (p^{\alpha}\Z)$.

Two $k$-DAO's are {\em equivalent} if they determine the same function $\Sigma_k^* \to \mathcal A$ given by $w \mapsto \tau(\delta(\overline{s}, w))$.
A $k$-DFAO with the fewest states in its equivalence class is said to be \emph{minimal}.
Given a $k$-DFAO $\mathcal M$, there is a unique minimal automaton equivalent to $\mathcal M$, up to renaming states~\cite[Theorem~3.10.1]{Shallit-SecondCourse}.
The minimal automaton can be computed by removing unreachable states and successively identifying pairs of equivalent states, that is, states such that starting in either state and reading a word produces the same output.
Note that this algorithm is typically described in the literature for $k$-DFA's --- automata with output alphabet $\{\textsf{True}, \textsf{False}\}$ --- but is generalized in a straightforward way to $k$-DFAO's.

A \emph{morphism} from a $k$-DAO $\mathcal M = (\mathcal S, \Sigma_k, \delta, \overline{s}, \mathcal A, \tau)$ to a $k$-DAO $\mathcal M' = (\mathcal S', \Sigma_k, \delta', \overline{s}', \mathcal A', \tau')$ consists of maps $\psi : \mathcal S \to \mathcal S'$ and $\pi: \mathcal A \to \mathcal A'$ such that $\psi(\overline{s}) = \overline{s}'$, $\psi(\delta(s, i)) = \delta'(\psi(s), i)$, and $\pi(\tau(s)) = \tau'(\psi(s))$ for all $s \in \mathcal S$ and $i \in \Sigma_k$.
In this article, the map $\pi$ will always be one of the projection maps $\pi_{\alpha, \beta}$.

Suppose that $(\mathcal M_\alpha)_{\alpha \geq 0}= (\mathcal S_\alpha, \Sigma_p, \delta_\alpha, \overline{s}_\alpha, \Z/(p^\alpha \Z), \tau_\alpha)_{\alpha\geq 0}$ is a sequence of finite automata such that, for each $0 \leq \alpha \leq \beta$, there is a morphism from $\mathcal M_{\beta}$ to $\mathcal M_\alpha$, consisting of a map $\psi_{\alpha, \beta} : \mathcal S_\beta \to \mathcal S_\alpha$ and the projection map $\pi_{\alpha, \beta}$, with the property that
\begin{equation}\label{commute}
\psi_{\alpha,\gamma} = \psi_{\alpha,\beta} \circ \psi_{\beta,\gamma}
\end{equation}
if $0 \leq \alpha \leq \beta \leq \gamma$. Then we call the sequence $(\mathcal M_\alpha)_{\alpha \geq 0}$ an {\em inverse family of finite automata}.
We remark that our automata have  output in $\Z/(p^\alpha \Z)$, because of our motivating examples, but a more general definition of an inverse family of finite automata is possible.

\begin{lemma}\label{inverse limit of automata}
Let $a \in \Z_p^\N$ be a sequence of $p$-adic integers such that $a \bmod p^\alpha$ is $p$-automatic for each $\alpha \geq 0$.
Let $\mathcal M_\alpha$ be the minimal automaton generating  $a \bmod p^\alpha$ in direct reading.
Then $(\mathcal M_\alpha)_{\alpha\geq 0}$ is an inverse family of finite automata.
\end{lemma}

\begin{proof}
Let $\mathcal M_\alpha = (\mathcal S_\alpha, \Sigma_p, \delta_\alpha, \overline{s}_\alpha, \Z/(p^\alpha \Z), \tau_\alpha)$, and let
$\mathcal M_\alpha'   = (\mathcal S_{\alpha +1}, \Sigma_p, \delta_{\alpha+1}, \overline{s}_{\alpha+1}, \Z/(p^\alpha \Z)   , \pi_{\alpha, \alpha+1} \circ \tau_{\alpha+1})$ be the automaton obtained by replacing the output function $\tau_{\alpha+1}$ in $\mathcal M_{\alpha+1}$ with the map $s \mapsto (\tau_{\alpha+1}(s) \bmod p^\alpha)$.
Then $\mathcal M_\alpha'$ and $\mathcal M_\alpha$ are equivalent.
Since $\mathcal M_\alpha$ is minimal, minimizing $\mathcal M_\alpha'$ gives $\mathcal M_\alpha$.
Since minimizing can be accomplished by successively identifying pairs of equivalent states,
each element of $\mathcal S_{\alpha + 1}$ can be assigned to a unique element of $\mathcal S_\alpha$; let $\psi_{\alpha, \alpha+1}$ be this map.
We claim that $\psi_{\alpha, \alpha + 1}$ and $\pi_{\alpha, \alpha + 1}$ comprise a morphism from $\mathcal M_{\alpha + 1}$ to $\mathcal M_\alpha$.
Minimizing an automaton maps the initial state to the initial state, so $\psi_{\alpha, \alpha+1}(\overline{s}_{\alpha +1})=\overline{s}_\alpha$.
Minimizing preserves edge relations, so $\psi_{\alpha, \alpha+1}(\delta_{\alpha+1}(s, i)) = \delta_\alpha(\psi_{\alpha, \alpha+1}(s), i)$ for every $s \in \mathcal S_{\alpha+1}$ and $i \in \Sigma_k$.
Finally, because of our choice of output function for $\mathcal M_\alpha'$, we have $\pi_{\alpha, \alpha+1}(\tau_{\alpha+1}(s)) = \tau_\alpha(\psi_{\alpha, \alpha+1}(s))$ by the equivalence of $\mathcal M_\alpha$ and $\mathcal M_\alpha'$.
Thus $\psi_{\alpha, \alpha+1}$ and $\pi_{\alpha, \alpha+1}$ comprise a morphism from $\mathcal M_{\alpha+1}$ to $\mathcal M_\alpha$.
Composing the maps $\psi_{\alpha, \alpha+1}$ gives maps $\psi_{\alpha, \beta}$ with the desired composition rule.
\end{proof}

\begin{example}\label{Catalan modulo 2}
We illustrate Lemma~\ref{inverse limit of automata} by computing the projection map $\psi_{1,2}$ for the sequence $C(n)_{n \geq 0}$ of Catalan numbers $C(n) = \frac{1}{n+1} \binom{2n}{n}$.
Catalan numbers are ubiquitous in combinatorial settings.
The generating function $y = \sum_{n \geq 0} C(n) x^n$ is algebraic and satisfies $x y^2 - y + 1 = 0$.
By Theorem~\ref{Denef--Lipshitz_1}, the sequence $(C(n) \bmod p^\alpha)_{n \geq 0}$ is $p$-automatic for each prime $p$ and each $\alpha \geq 0$.
Let $p = 2$.
For $\alpha = 2$, the automaton $\mathcal M_2 = (\mathcal S_2, \Sigma_2, \delta_2, \overline{s}_2, \Z/ (4 \Z), \tau_2)$ appears in Example~\ref{Catalan modulo 4}.
It produces the sequence
\[
	C_2 = (C(n) \bmod 4)_{n \geq 0} = 1, 1, 2, 1, 2, 2, 0, 1, 2, 2, 0, 2, 0, 0, 0, 1, \dots.
\]
The automaton in Example~\ref{substitution construction} is $(\mathcal S_2, \Sigma_2, \delta_2, \overline{s}_2, \mathcal S_2, \text{id})$, where the output function is the identity map and $s_0 = \overline{s}_2$.
The sequence $u_2$ generated by this automaton is
\[
	u_2 = s_0, s_1, s_2, s_3, s_2, s_4, s_5, s_3, s_2, s_4, s_5, s_4, s_5, s_5, s_5, s_3, \dots.
\]
For $\alpha = 1$, the sequence
\[
	C_1 = (C(n) \bmod 2)_{n \geq 0} = 1, 1, 0, 1, 0, 0, 0, 1, 0, 0, 0, 0, 0, 0, 0, 1, \dots
\]
is generated by $\mathcal M_1 = (\mathcal S_1, \Sigma_2, \delta_1, \overline{s}_1, \Z/ (2 \Z), \tau_1)$, which is the following.
\begin{center}
	\includegraphics{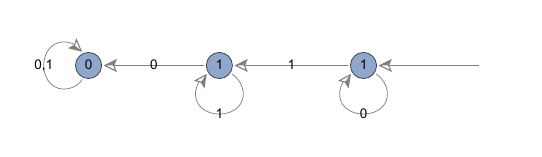}
\end{center}
The automaton $(\mathcal S_1, \Sigma_2, \delta_1, \overline{s}_1, \mathcal S_1, \text{id})$ is as follows, where we name the states $t_0, t_1, t_2$ and $t_0 = \overline{s}_1$.
\begin{center}
	\includegraphics{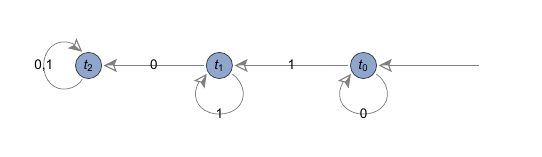}
\end{center}
The states $s_2, s_4, s_5$ in $\mathcal M_2$ correspond to output congruent to $0$ or $2$ modulo $4$; since there is only one state in $\mathcal M_1$ whose output is $0$ modulo $2$, it follows that
\[
	\psi_{1,2}(s_2) = \psi_{1,2}(s_4) = \psi_{1,2}(s_5) = t_2.
\]
One checks that $\psi_{1,2}(s_1) = \psi_{1,2}(s_3) = t_1$, and on the initial state we have $\psi_{1,2}(s_0) = t_0$.
\end{example}

We now give the definition of a profinite automaton. Readers familiar with inverse limits will recognize profinite automata as inverse limits in the category of $p$-DAO's with automaton morphisms.

\begin{definition}\label{profinite_automaton}
Let  $(\mathcal M_\alpha)_{\alpha\geq 0} = (\mathcal S_\alpha, \Sigma_p, \delta_\alpha, \overline{s}_\alpha, \Z/(p^\alpha \Z), \tau_\alpha)_{\alpha \geq 0}$ be an inverse family of finite automata with maps $\psi_{\alpha,\beta}:\mathcal S_\beta \rightarrow \mathcal S_\alpha$  and  projection maps $\pi_{\alpha,\beta}$.
We call $\mathcal M=(\mathcal S, \Sigma_p, \delta, \overline{s}, \Z_p,\tau)$   a  \emph{profinite automaton} if
\begin{enumerate}
\item\label{profinite_automaton:states}
$\mathcal S = \{ (s_\alpha)_{\alpha\geq 0}\in \prod_{\alpha \geq 0} \mathcal S_\alpha:    \psi_{\alpha,\beta}( s_{\beta}  )=s_{\alpha}  \mbox{ for each } 0 \leq \alpha \leq \beta   \}$,
\item\label{profinite_automaton:initial}
 $(\overline{s})_\alpha= \overline{s}_\alpha$ for each $\alpha \geq 0$,
\item\label{profinite_automaton:delta}
$(\delta(s,i))_\alpha = \delta_{\alpha}(s_\alpha,i)$ for each $\alpha \geq 0$, and
\item\label{profinite_automaton:tau}
$\pi_{\alpha, \infty}(\tau(s)) = \tau_\alpha(s_\alpha)$ for each $\alpha \geq 0$.
\end{enumerate}
\end{definition}

Note that the ranges of the maps $\delta$ and $\tau$  are in $\mathcal S$ and $\Z_p$ respectively because the maps $\psi_{\alpha, \beta}$ satisfy   Equation~\eqref{commute}. 
We define $\psi_{\alpha, \infty} : \mathcal S \to \mathcal S_\alpha$ by $\psi_{\alpha, \infty}(s) \colonequal s_\alpha$.
Profinite automata satisfy the following commutative diagram.
\[
\xymatrix{
	\mathcal S \ar[rd]|-{\psi_{\beta,\infty}} \ar[rrd]|-{\psi_{\alpha,\infty}} \ar[ddd]_\tau \\
																	& \mathcal S_\beta \ar[r]_{\psi_{\alpha,\beta}} \ar[d]_{\tau_\beta}	& \mathcal S_\alpha \ar[d]_{\tau_\alpha} \\
																	& \Z/(p^\beta \Z) \ar[r]^{\pi_{\alpha,\beta}}						& \Z/(p^\alpha \Z) \\
	\Z_p \ar[ru]|-{\pi_{\beta,\infty}} \ar[rru]|-{\pi_{\alpha,\infty}}
}
\]

One can alternatively define profinite automata using reverse reading rather than direct reading.

By Theorem~\ref{Denef--Lipshitz_1} and Lemma~\ref{inverse limit of automata}, algebraic sequences of $p$-adic integers are generated by profinite $p$-DAO's. But also, Theorem~\ref{Denef--Lipshitz_1}   has a converse in \cite[Theorem 3.1]{Denef--Lipshitz}: a sequence generated by a profinite $p$-DAO is, for each $\alpha$, congruent modulo $p^{\alpha}$ to an algebraic sequence of $p$-adic integers.

\begin{example}
The following shows the first few states in a profinite $2$-DAO generating the sequence of Catalan numbers.
This automaton is equivalent modulo $4$ to the automaton $\mathcal M_2$ in Example~\ref{Catalan modulo 4} and equivalent modulo $2$ to the automaton $\mathcal M_1$ in Example~\ref{Catalan modulo 2}.
\begin{center}
	\includegraphics[scale=.7]{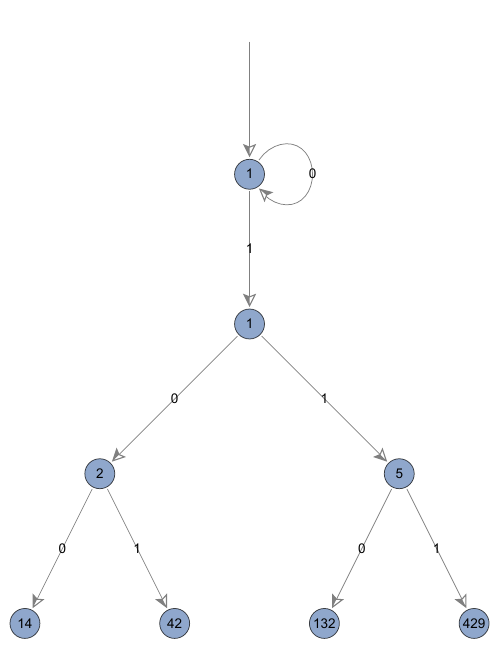}
\end{center}
\end{example}

A profinite automaton can have uncountably many states, since in general each element of $\tau(\mathcal S)$ corresponds to an infinite path in a rooted tree that encodes the residues attained by a sequence modulo each $p^\alpha$.
Again let $C(n) = \frac{1}{n+1} \binom{2n}{n}$ be the $n$th Catalan number, and let $p = 2$.
Consider the rooted tree in which the vertices on level $\alpha$ consist of all residues $j$ modulo $2^\alpha$ such that $C(n) \equiv j \mod 2^\alpha$ for some $n \geq 0$.
Two vertices on adjacent levels are connected by an edge if one residue projects to the other.
Levels $0$ through $5$ of this tree are shown below.
In particular, $C(n) \nequiv 3 \mod 4$ for all $n \geq 0$; therefore there are only $3$ vertices on level $\alpha = 2$.
\begin{center}
	\includegraphics{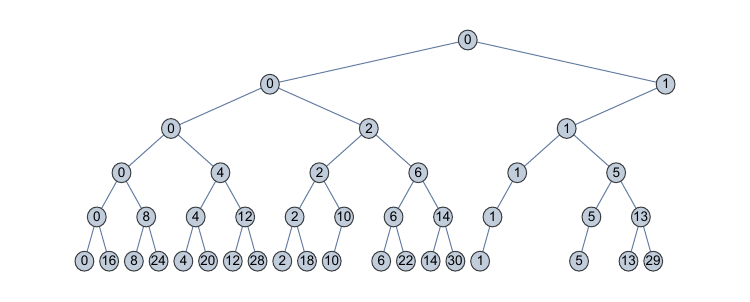}
\end{center}
The residues attained by the sequence of Catalan numbers modulo $2^\alpha$ can be computed automatically~\cite[Section~3.1]{Rowland--Yassawi}.
The number of states in the minimal automaton $\mathcal M_\alpha$ is bounded below by the number of vertices on level $\alpha$.

We will henceforth use the notation $\varprojlim$ to denote the inverse limit of an inverse family. Thus  for example, we shall write $\mathcal M=\varprojlim \mathcal M_\alpha$ to denote the inverse limit of an inverse family of minimal automata, and $s =\varprojlim s_\alpha$ to denote an element in the state set $\mathcal S$.

We say that a $p$-DAO is \emph{minimal} if it contains no unreachable states and no pair of equivalent states.
For finite automata, this definition of minimality coincides with the earlier definition.
Up to renaming of states, profinite minimal automata have the characterization given by Theorem~\ref{profinite characterization} below.
For a $p$-DAO $\mathcal M = (\mathcal S, \Sigma_p, \delta, \overline{s}, \Z_p, \tau)$, let us use $\mathcal M \bmod p^\alpha$ as an abbreviation for $(\mathcal S, \Sigma_p, \delta, \overline{s}, \Z/(p^\alpha \Z), \pi_{\alpha, \infty} \circ \tau)$.
Let $\mathcal M_\alpha = (\mathcal S_\alpha, \Sigma_p, \delta_\alpha, \overline{s}_\alpha, \Z/(p^\alpha \Z), \tau_\alpha)$ be the minimal  automaton equivalent to $\mathcal M \bmod p^\alpha$.
For each $0 \leq \alpha \leq \beta$, the automaton $\mathcal M_\beta$ projects modulo $p^\alpha$ to $\mathcal M_\alpha$.
We say that $\mathcal S$ is \emph{closed under inverse limits} if every inverse family $(s_\alpha)_{\alpha \geq 0}$, where $s_\alpha \in \mathcal S_\alpha$, has an inverse limit $s \in \mathcal S$.
That is, for every sequence of states $(s_\alpha)_{\alpha \geq 0}$ where $s_\alpha \in \mathcal S_\alpha$ and $\pi_{\alpha, \beta}(\tau_\beta(\delta_\beta(s_\beta, w))) = \tau_\alpha(\delta_\alpha(s_\alpha, w))$ for all $0 \leq \alpha \leq \beta$ and $w \in \Sigma_p^*$, there exists a state $s \in \mathcal S$ such that $\pi_{\alpha, \infty}(\tau(\delta(s, w))) = \tau_\alpha(\delta_\alpha(s_\alpha, w))$ for all $\alpha \geq 0$ and for all $w \in \Sigma_p^*$.

\begin{example}
Let $p = 2$, and for each $\alpha \geq 0$ let $\iota : \Z/(2^\alpha \Z) \to \Z_2$ be the lifting defined by $\iota(n + 2^\alpha \Z) = n + 0 \cdot 2^\alpha + 0 \cdot 2^{\alpha + 1} + \cdots$ for each $0 \leq n \leq 2^\alpha - 1$.
Again let $C(n)$ denote the $n$th Catalan number.
For $n_\ell \cdots n_1 n_0 \in \Sigma_2^*$, let $\val_2(n_\ell \cdots n_1 n_0) = n_\ell 2^\ell + \cdots + n_1 2 + n_0$.
Consider the minimal $2$-DAO $\mathcal M = (\mathcal S, \Sigma_2, \delta, \overline{s}, \Z_2, \tau)$ defined for $w \in \Sigma_2^*$ by
\[
	\tau(\delta(\overline{s}, w)) =
	\begin{cases}
		\iota(C(\val_2(v)) \bmod 2^\alpha)	& \text{if $w = 1^\alpha 0 v$ for some $\alpha \geq 1$ and $v \in \Sigma_2^*$} \\
		0							& \text{otherwise.}
	\end{cases}
\]
Since $1^\alpha 0$ is a prefix of $w$ for at most one value of $\alpha$, the automaton $\mathcal M$ is not over-determined.
The set $\mathcal S$ is not closed under inverse limits, since, for each $\alpha \geq 0$, the sequence $(C(n) \bmod 2^\alpha)_{n \geq 0}$ occurs in $\mathcal M_\alpha$ (i.e.\ is obtained by reading words $v$ starting from the state $\delta_\alpha(\overline{s}_\alpha, 1^\alpha 0)$), but $C(n)_{n \geq 0}$ does not occur in $\mathcal M$.
This example shows that the condition that $\mathcal S$ is closed under inverse limits is necessary in the following theorem.
\end{example}

\begin{theorem}\label{profinite characterization}
Let $\mathcal M = (\mathcal S, \Sigma_p, \delta, \overline{s}, \Z_p, \tau)$ be a minimal $p$-DAO.
Then $\mathcal M$ is profinite if and only if
\begin{enumerate}
\item
the automaton $(\mathcal S, \Sigma_p, \delta, \overline{s}, \Z/(p^\alpha \Z), \pi_{\alpha, \infty} \circ \tau)$ is equivalent to a finite automaton for each $\alpha \geq 0$, and
\item
$\mathcal S$ is closed under inverse limits.
\end{enumerate}
\end{theorem}

\begin{proof}
The forward direction is clear:
Since $\mathcal M$ is profinite, it has an inverse family $(\mathcal M_\alpha)_{\alpha \geq 0}$ of finite automata, and $\mathcal M \bmod p^\alpha$ is equivalent to $\mathcal M_\alpha$.
Moreover, $\mathcal S$ is closed under inverse limits by Part~\ref{profinite_automaton:states} of Definition~\ref{profinite_automaton}.

For the other direction, let $\mathcal M_\alpha = (\mathcal S_\alpha, \Sigma_p, \delta_\alpha, \overline{s}_\alpha, \Z/(p^\alpha \Z), \tau_\alpha)$ be the minimal finite automaton equivalent to $\mathcal M \bmod p^\alpha$.
For each $0 \leq \alpha \leq \beta$, the automaton $\mathcal M_\beta$ projects modulo $p^\alpha$ to $\mathcal M_\alpha$, so there is a unique map $\psi_{\alpha, \beta} : \mathcal S_\beta \to \mathcal S_\alpha$ such that $\psi_{\alpha, \beta}$ and $\pi_{\alpha, \beta}$ comprise a morphism from $\mathcal M_\beta$ to $\mathcal M_\alpha$.
We have $\psi_{\alpha,\gamma} = \psi_{\alpha,\beta} \circ \psi_{\beta,\gamma}$ if $0 \leq \alpha \leq \beta \leq \gamma$, so $(\mathcal M_\alpha)_{\alpha \geq 0}$ is an inverse family.
Let $\mathcal M' = (\mathcal S', \Sigma_p, \delta', \overline{s}', \Z_p, \tau')$ be the profinite automaton corresponding to the inverse family $(\mathcal M_\alpha)_{\alpha \geq 0}$.
We establish a map $\psi : \mathcal S \to \mathcal S'$ such that $\psi$ and the identity map $\text{id} : \Z_p \to \Z_p$ comprise a morphism from $\mathcal M$ to $\mathcal M'$.
Define the image $\psi(s)$ of a state $s \in \mathcal S$ as follows.
For each $\alpha \geq 0$, consider the function $\Sigma_p^* \to \mathcal \Z/(p^\alpha \Z)$ defined on words by $w \mapsto \pi_{\alpha, \infty}(\tau(\delta(s, w)))$, i.e.\ obtained by modifying $\mathcal M \bmod p^\alpha$ to have initial state $s$.
Since $\mathcal M \bmod p^\alpha$ is equivalent to $\mathcal M_\alpha$, this function is also given by $w \mapsto \tau_\alpha(\delta_\alpha(s_\alpha, w))$ for some state $s_\alpha \in \mathcal S_\alpha$; namely, since $\mathcal M$ has no unreachable states, there exists a word $v \in \Sigma_p^*$ such that $\delta(\overline{s}, v) = s$, and we can take $s_\alpha = \delta_\alpha(\overline{s}_\alpha, v)$.
Moreover, $s_\alpha$ is unique, since two distinct states in $\mathcal M_\alpha$ yielding the same function would be equivalent and would contradict the minimality of $\mathcal M_\alpha$.
The sequence $(s_\alpha)_{\alpha \geq 0}$ of states forms an inverse family, and we define $\psi(s) = \varprojlim \delta_\alpha(\overline{s}_\alpha, v) \in \mathcal S'$.
The initial state $\overline{s}$ projects to the initial state $\overline{s}_\alpha$ for each $\alpha \geq 0$; by Part~\ref{profinite_automaton:initial} of Definition~\ref{profinite_automaton}, $\overline{s}'$ also projects to $\overline{s}_\alpha$ for each $\alpha \geq 0$, so $\psi(\overline{s}) = \varprojlim \overline{s}_\alpha = \overline{s}'$.
Moreover, for each $s \in \mathcal S$ and $i \in \Sigma_p$ we have
\[
	\psi(\delta(s, i))
	= \varprojlim \delta_\alpha(\overline{s}_\alpha, vi)
	= \varprojlim \delta_\alpha(\delta_\alpha(\overline{s}_\alpha, v), i)
	= \delta'(\psi(s), i)
\]
for the appropriate word $v$,  where $vi$ is the word $v$ with $i$ appended, and where the last equality follows from Part~\ref{profinite_automaton:delta} of Definition~\ref{profinite_automaton}.
Finally, for each $s \in \mathcal S$ we have
\[
	\tau(s)
	= \varprojlim \tau_\alpha(\delta_\alpha(\overline{s}_\alpha, v))
	= \tau'(\psi(s)),
\]
for the appropriate word $v$, where the last equality follows from Part~\ref{profinite_automaton:tau} of Definition~\ref{profinite_automaton}.
Therefore $\psi$ and $\text{id}$ comprise a morphism from $\mathcal M$ to $\mathcal M'$.

It remains to show that $\psi$ is bijective.
Suppose $s, t \in \mathcal S$ such that $\psi(s) = \psi(t)$.
Let $v, z$ be words such that $\delta(\overline{s}, v) = s$ and $\delta(\overline{s}, z) = t$.
Since $\psi(s) = \psi(t)$, we have $\delta_\alpha(\overline{s}_\alpha, v) = \delta_\alpha(\overline{s}_\alpha, z)$ for each $\alpha \geq 0$.
For all $w \in \Sigma_p^*$, we have
\[
	\tau(\delta(s, w))
	= \varprojlim \tau_\alpha(\delta_\alpha(\overline{s}_\alpha, vw))
	= \varprojlim \tau_\alpha(\delta_\alpha(\overline{s}_\alpha, zw))
	= \tau(\delta(t, w)).
\]
Therefore $s$ and $t$ are equivalent states in $\mathcal M$.
Since $\mathcal M$ is minimal, it follows that $s = t$, and hence $\psi$ is injective.
To show surjectivity, let $s' \in \mathcal S'$.
The state $s'$ projects to a state $s_\alpha \in \mathcal S_\alpha$ for each $\alpha \geq 0$.
Since $\mathcal S$ is closed under inverse limits, there exists $s = \varprojlim s_\alpha \in \mathcal S$ such that $\psi(s) = s'$.
Therefore $\mathcal M$ and $\mathcal M'$ are isomorphic, so $\mathcal M$ is profinite.
\end{proof}

\section{Substitution shifts defined by profinite automata}\label{dynamical properties}

In this section we construct a substitution shift,  usually on an uncountable alphabet, from a profinite automaton.
The construction is completely analogous to the construction of  substitution shifts  on a finite alphabet. We generalize known dynamical properties of constant-length substitution shifts on a finite alphabet to these substitution shifts on an uncountable alphabet.

\subsection{Profinite shifts and substitutions}

Let $\mathcal A$ be a compact topological space.
We endow $\mathcal A^\N$ 
with the product topology, so that it is a compact space, and let 
$\sigma:\mathcal A^\N \rightarrow \mathcal A^\N$ 
denote the {\em shift map} $\sigma (  a(n)_{n \geq 0}   )\colonequal a(n+1)_{n \geq 0}$; $\sigma$ is a continuous mapping on $\mathcal A^\N$.
If $ a = a(n)_{n \geq 0}  \in {\mathcal A}^\N$,  define  $X_{  a} \colonequal \overline{\{\sigma^{n}(  a): n\in \N\}}$, the closure of the $\sigma$-orbit of $a$ in $\mathcal A^\N$;
we  call 
$(X_a, \sigma)$ the {\em one-sided  shift generated by $a$}.
We will be primarily interested in {\em substitution shifts}, namely, those where $a$ is a fixed point of a substitution defined on an alphabet $\mathcal A$, and their codings.

 Traditionally, $\mathcal A$ is a finite alphabet, with the discrete topology.
We will also consider shifts defined on non-finite alphabets. Some of these alphabets will be subsets of $\Z_p$, with the topology defined by the $p$-adic absolute value, which coincides with the inverse limit topology if we consider $\Z_p$ as the inverse limit of the topological rings $\Z/(p^\alpha \Z)$.
If $\mathcal A$ is a closed subset of $\Z_p$, then $\mathcal A$  inherits the subspace topology from $\Z_p$. 
We extend the projection maps $\pi_{\alpha, \beta}$ and $\pi_{\alpha, \infty}$ termwise to sequences in $(\Z / (p^\beta \Z))^\N$ and $\Z_p^\N$.
If $a \in \Z_p^\N$ we let  $  a_\alpha \colonequal \pi_{\alpha, \infty} (a)$ be the sequence obtained by reducing each term modulo $p^\alpha$.
We also extend output functions $\tau$ and morphism maps $\psi$ to sequences termwise.

 For the rest of the article we continue to use the  notation of  Definition~\ref{profinite_automaton}.  
 Namely, we suppose that we have a sequence  $a \in \Z_p^\N$  of $p$-adic integers   such that $a_\alpha $ is $p$-automatic for each $\alpha\geq 0$. We let $\mathcal M_\alpha = (\mathcal S_\alpha, \Sigma_{p}, \delta_\alpha, \overline{s}_{ \alpha}, \Z/(p^\alpha \Z), \tau_\alpha)$ be the minimal automaton generating $a_\alpha$.  Then Lemma~\ref{inverse limit of automata} tells us that $(\mathcal M_\alpha)_{\alpha \geq 0}$ is an inverse family of finite automata. 
 For $0\leq \alpha\leq \beta$, we let the morphism from $\mathcal M_\beta$ to $\mathcal M_\alpha$ be given by the map $\psi_{\alpha, \beta}:\mathcal S_\beta \rightarrow \mathcal S_\alpha$ and the projection $\pi_{\alpha,\beta}$. 
Finally, we let $u_\alpha \in \mathcal S_\alpha^\N$ be the sequence generated by   $(\mathcal S_\alpha, \Sigma_{p},\delta_\alpha, \overline{s}_{ \alpha}, \mathcal S_\alpha, \text{id})$, 
so that $a_\alpha = \tau_\alpha(u_\alpha)$.

The following theorem extends the results of  Lemma~\ref{inverse limit of automata} to shifts.

\begin{theorem}\label{commuting_diagram_lemma_2}
Suppose that $a \in \Z_p^\N$ is a sequence of $p$-adic integers   such that $a_\alpha $ is $p$-automatic for each $\alpha\geq 0$.
Then there is a sequence $u$, generated by a profinite automaton, and a letter-to-letter projection $\tau:X_u\rightarrow X_a$ with
$\tau(u)=a$,  such that  for each $0 \leq \alpha \leq \beta$ the following diagram commutes.

\[
\xymatrix{
	(X_u, \sigma) \ar[rd]|-{\psi_{\beta,\infty}} \ar[rrd]|-{\psi_{\alpha,\infty}} \ar[ddd]_\tau \\
																	& (X_{u_\beta}, \sigma) \ar[r]_{\psi_{\alpha,\beta}} \ar[d]_{\tau_\beta}		& (X_{u_\alpha}, \sigma) \ar[d]_{\tau_\alpha} \\
																	& (X_{a_\beta}, \sigma) \ar[r]^{\pi_{\alpha,\beta}}					& (X_{a_\alpha}, \sigma) \\
	(X_a, \sigma) \ar[ru]|-{\pi_{\beta,\infty}} \ar[rru]|-{\pi_{\alpha,\infty}}
}\]

\end{theorem}

\begin{proof}
Let $0\leq \alpha\leq \beta$. As the maps $\psi_{\alpha,\beta}$, $\pi_{\alpha, \beta}$, and $\tau_\beta$ are extended termwise to sequences, they commute with the shift.
Since
\[\psi_{\alpha,\beta} (u_{\beta}) = \psi_{\alpha,\beta}(\delta_{\beta}(\overline{s}_\beta, \rep_p(k))_{k\geq 0})= \delta_{\alpha}(\overline{s}_\alpha, \rep_p(k))_{k\geq0}= u_\alpha,\]
it follows that $\psi_{\alpha,\beta}$ maps $X_{u_\beta}$ to $ X_{u_\alpha}$.
That the map $\pi_{\alpha,\beta}$ sends $X_{a_\beta}$ to $X_{a_\alpha}$  is proved analogously. Finally, as $\pi_{\alpha, \beta}\circ \tau_{\beta} = \tau_\alpha\circ \psi_{\alpha,\beta}$ is true on the orbit of $u_{\beta}$, it extends to an equality on $X_{u_\beta}$. Thus the inner displayed diagram commutes.

Define 
\[X\colonequal \varprojlim X_{u_\alpha} = \left\{x = \varprojlim x_\alpha:  x_\alpha \in     X_{u_\alpha} \mbox{ and } \psi_{\alpha,\beta}(x_{\beta}) = x_\alpha \mbox{ for each  } 0 \leq \alpha \leq \beta 
  \right\}.\]
 The space  $X$ lives in $\prod_{\alpha \geq 0} X_{u_\alpha}$ and inherits the subspace topology from the product topology on $\prod_{\alpha \geq 0} X_{u_\alpha}$. 
Since   for each $\alpha\geq 0$, $X_{u_\alpha}$ is Hausdorff and $\sigma$-invariant, then  is $X$ closed and $\sigma$-invariant.  Define $u \colonequal \varprojlim u_\alpha$;
 then $u\in X$,
  and so $X_u\subset X$.   

To show that $X\subset X_u$,  let $x=\varprojlim x_\alpha \in X$. We will show that for any open set $U$ containing $x$, there is some $n\in \N$ with $\sigma^n(u)\in U$.
Since each $X_{u_\alpha}$ is a shift defined on a finite alphabet, it is metrizable; denote this metric by $d$, where points are close if and only if they agree on a sufficiently large initial block.
     Any open set $U$ containing $x$ contains a set of the form $B=\left(B_0(x_0, \epsilon) \times \dots  \times B_\beta(x_\beta, \epsilon)  \times \prod_{\gamma >\beta}X_{u_\gamma}\right) \cap X$ where $B(x,\epsilon) $ is a ball of radius $\epsilon$ centered at $x$. 
     Since $x_\beta \in X_{u_\beta}$, there is some $n$ such that $\sigma^n(u_\beta)\in B_\beta (x_\beta, \epsilon)\cap X_{u_\beta}$. If $\alpha\leq \beta$, then $d(\sigma^n(u_\alpha), x_\alpha)= d(\sigma^n \psi_{\alpha,\beta}(u_\beta), \psi_{\alpha,\beta}(x_\beta))= d(\psi_{\alpha,\beta}\sigma^n (u_\beta), \psi_{\alpha,\beta}(x_\beta))< \epsilon$.       
  In other words, $\sigma^n(u) \in B$, and it follows that $X\subset X_u$.

For $x \in X_u$, define $\tau(x) \colonequal  \varprojlim \tau_\alpha(x_\alpha)$. Let $a \colonequal \tau(u)$; then similar to above, the set 
 $Y: = \varprojlim \tau_\alpha (X_{u_\alpha} )$ 
 is closed and $\sigma$-invariant, and $Y= X_a$.  Now the verification that the diagram commutes follows from the definitions of $X=X_u$ and $Y=X_a$.
\end{proof}

We remark that if $u_\alpha$ is the sequence generating by $\mathcal M_\alpha$, and $\mathcal M$ is the profinite automaton defined by $(\mathcal M_\alpha)_{\alpha \geq 0}$ with state set $\mathcal S$, then the sequence $u=\varprojlim u_\alpha$ can be identified with an element of $\mathcal S^\mathbb N$, where the $n$th component is the inverse limit of the $n$th component of $u_\alpha$.
Since the interpretation will be clear from context, we pass between these two objects without further remarks, just as we identify elements of $\Z_p^\N$ with elements of $\prod_{\alpha \geq 0} (\Z/(p^\alpha \Z))^\N$.

Recall that  a finite automaton defines a substitution as described after Theorem~\ref{Cobham}.

\begin{theorem}\label{substitution_theorem}
Let $(\mathcal M_\alpha)_{\alpha\geq 0} = (\mathcal S_\alpha, \Sigma_{p}, \delta_\alpha, \overline{s}_{\alpha}, \Z/(p^\alpha \Z), \tau_\alpha)_{\alpha \geq 0}$ be an inverse family of finite automata whose inverse limit is the  profinite automaton $\mathcal M=(\mathcal S, \Sigma_p, \delta, \overline{s}, \Z_p,\tau)$.
For each $\alpha \geq 0$, let $\theta_\alpha$ be the  substitution on $\mathcal S_\alpha$ corresponding to $\mathcal M_\alpha$, with fixed point $u_\alpha$.
Let  $u=\varprojlim u_\alpha \in \mathcal S^{\mathbb N}$. 
Then there is a  length-$p$ substitution  $\theta:\mathcal S\rightarrow \mathcal S^p$, with $\theta(u)=u$, where for each $0 \leq \alpha \leq \beta$ the following diagram commutes.
\[
\xymatrix{
	X_u \ar[rd]|-{\psi_{\beta,\infty}} \ar[rrd]|-{\psi_{\alpha,\infty}} \ar[ddd]_\theta \\
																& X_{u_\beta} \ar[r]_{\psi_{\alpha,\beta}} \ar[d]_{\theta_\beta}	& X_{u_\alpha} \ar[d]_{\theta_\alpha} \\
																& X_{u_\beta} \ar[r]^{\psi_{\alpha,\beta}}					& X_{u_\alpha} \\
	X_u \ar[ru]|-{\psi_{\beta,\infty}} \ar[rru]|-{\psi_{\alpha,\infty}}
}
\]
\end{theorem}
 
\begin{proof}

 For each  $s=\varprojlim s_\alpha\in  \mathcal S$, we will define $\theta (s)$, a word of length $p$. If $0\leq i\leq p-1$, we  use $\theta(s)_i$ to refer to the $i$-th letter of $\theta(s)$. Define 
 \[ \theta (s)_i\colonequal \varprojlim \theta_\alpha (s_\alpha)_i .\]
Then for $k\geq 0$ and $0\leq i \leq p-1$,
\[  u(pk+i)= \varprojlim u_\alpha(pk+i) = \varprojlim  \theta_\alpha  u_\alpha(k)_i = \theta(u(k))_i,    \]
which implies that $\theta(u)=u$.
As in Definition~\ref{substitution}, the substitution $\theta :\mathcal S \rightarrow \mathcal S^p$ extends to the substitution $\theta:X_u\rightarrow X_u$.
To verify that the inner diagram commutes, we note that
\[
	\psi_{\alpha, \beta}\theta_\beta \sigma^j (u_\beta)
	= \psi_{\alpha, \beta}\sigma^{jp} (u_\beta)
	= \sigma^{jp} (u_\alpha)
	= \theta_\alpha \sigma^j (u_\alpha)
	= \theta_\alpha \psi_{\alpha,\beta} \sigma^j (u_\beta),
\]
and the continuity of $\psi_{\alpha, \beta}$ and $\theta_\beta$ gives us  the result. The commutativity of the outer diagram follows in a similar way.
\end{proof}

We call the substitution $\theta$ of Theorem~\ref{substitution_theorem} a {\em profinite substitution}, and  the shift   $(X_u, \sigma)$ of  Theorem~\ref{commuting_diagram_lemma_2}  a {\em profinite substitution shift.}

\subsection{Dynamical properties of profinite substitution shifts}

Recall that   the substitution $\theta$ on the finite alphabet $\mathcal A$ is {\em primitive} if there is some 
$k\in \N$ such that for all $a,a'\in \mathcal A$,
the word $\theta^k(a)$ contains $a'$.  For primitive substitutions on a finite alphabet,  any word that appears in any $x \in X_u$ appears in $x$ with bounded gaps~\cite[Chapter 5.2]{queffelec}, and $X_u$ is {\em minimal}: $X_u$ is the $\sigma$-orbit closure of any point in $X_u$. An application of the argument in the second part of the proof of  Theorem~\ref{commuting_diagram_lemma_2}  yields the following.

\begin{proposition}\label{minimal}
Let $u$ be a fixed point of a profinite substitution $\theta=\varprojlim \theta_\alpha$. 
If $\theta_\alpha$ is length-$p$ and  primitive for each $\alpha\geq 0$ then $(X_u,\sigma)$ is minimal.
\end{proposition}

\begin{definition}[{\cite[Definitions 5 and 6]{choksi}}] 
 Let 
  $\{(X_{\alpha}, \mathcal B_\alpha, \mu_\alpha): \alpha\geq 0\}$ be a family of compact Borel measure spaces, and for each $0\leq \alpha\leq \beta$ let  $\psi_{\alpha,\beta}:X_{\beta}\rightarrow X_{\alpha}$ be a continuous map.  If $\mu_\alpha(A) = \mu_{\beta}(\psi_{\alpha, \beta}^{-1}(A))$ for each  $A\in \mathcal B_{\alpha}$ and each $0\leq \alpha\leq \beta$,
 then  $\{(X_{\alpha}, \mathcal B_\alpha, \mu_\alpha):\alpha\geq 0 \}$  is called an  {\em inverse family of (compact) topological measure spaces}.
 
Given an inverse family of compact topological measure spaces, let
\[X\colonequal  \{ x=\varprojlim x_\alpha : x_\alpha \in X_\alpha \mbox{ and }  \psi_{\alpha,\beta} (x_\beta)=x_\alpha \mbox{ for each } 0\leq \alpha\leq \beta  \}.\]
Let    $\psi_{\alpha, \infty}: X\rightarrow X_\alpha$ be the natural projection map, and  let  $\mathcal B_\alpha^*\colonequal \psi_{\alpha,\infty}^{-1}(\mathcal B_\alpha)$. Let $\mathcal B$ be the $\sigma$-algebra generated by  $\bigcup_{\alpha} \mathcal B_\alpha^*$.  Then  $\mu (B)\colonequal \mu_\alpha (\psi_{\alpha, \infty}B)$ for  $B \in \mathcal B_\alpha^*$ defines a finitely additive set function on $\bigcup_{\alpha} \mathcal B_\alpha^*$. If $\mu$ has an extension to $\mathcal B$, then we say that $(X, \mathcal B, \mu)$ is the {\em inverse limit} of $\{(X_{\alpha}, \mathcal B_\alpha, \mu_\alpha): \alpha\geq 0\}$.
\end{definition}

\begin{theorem}[{\cite[Theorems 2.2 and  2.3]{choksi}}]\label{choksi_theorem}
Let   $\{(X_{\alpha}, \mathcal B_\alpha, \mu_\alpha):   \alpha \geq 0    \} $ be an inverse family of compact topological measure spaces. Then 
 the inverse limit $(X, \mathcal B,\mu)$ exists. If $\mu_\alpha$ is Baire for each $\alpha$, then $\mu$ is Baire.
\end{theorem}

Let $\theta$ be a primitive substitution on a finite alphabet. Then any $\theta$-fixed point $u$, (or any $\theta$-periodic point $u$ in the case where there are no $\theta$-fixed points)  generates the  substitution shift $(X_u,\sigma)$, which is independent of the fixed point chosen. With the Borel $\sigma$-algebra $\mathcal B$, this shift is uniquely ergodic~\cite{michel}.

\begin{corollary}\label{choksi_corollary}
Let $p$ be prime. 
Let $u$ be a fixed point of a profinite substitution $\theta=\varprojlim \theta_\alpha$. 
Suppose that for each $\alpha\geq 0$, the substitution $\theta_\alpha$ is length-$p$ and primitive. 
Then $(X_u, \sigma)$ is uniquely ergodic.
 \end{corollary}

\begin{proof}
Let $u_\alpha$ be the fixed point of $\theta_\alpha$ such that   $\{(X_{u_\alpha},\sigma) : \alpha\geq 0\}$  is the inverse family  of substitution shifts generated by $(\theta_\alpha)_{\alpha\geq 0}$.
 As $\theta_\alpha$ is primitive,  we let $\mu_\alpha$ denote the unique $\sigma$-invariant measure for $(X_{u_\alpha}, \mathcal B_\alpha, \sigma   )$ with $\mathcal B_\alpha$ the Borel $\sigma$-algebra.   Then 
$\mu_\alpha= \mu_{\beta}\circ\pi_{\alpha, \beta}^{-1}$ on cylinder sets,  so that $\mu_\alpha= \mu_{\beta}\circ\pi_{\alpha, \beta}^{-1}$, i.e.\ $\{(X_{u_\alpha}, \mathcal B_\alpha, \mu_\alpha) : \alpha \geq 0\} $ is 
an inverse family of topological measure spaces. We apply 
  Theorem~\ref{choksi_theorem} to the inverse family   $\{(X_{u_\alpha}, \mathcal B_\alpha,  \mu_\alpha, \sigma): \alpha\geq 0\}$ 
    to conclude  that the inverse limit $\mu=\varprojlim \mu_\alpha$ exists on the $\sigma$-algebra $\mathcal B$  generated by  $\bigcup_{\alpha} \mathcal B_\alpha^*$,  with 
  $\mathcal B_\alpha^*\colonequal \psi_{\alpha,\infty}^{-1}(\mathcal B_\alpha)$. 
   Furthermore since each $\mu_\alpha$ is $\sigma$-invariant, then  $\mu$ is $\sigma$-invariant on each $\mathcal B_\alpha^*$ and this implies that $\mu$ is 
$\sigma$-invariant on $\mathcal B$.

 Let $\nu$ be any other $\sigma$-invariant measure on $(X_u, \mathcal B)$. Then for any $\alpha$,
 $\nu_\alpha (A)\colonequal\nu (\psi_{\alpha, \infty}^{-1}(A))$, $A \in \mathcal B_\alpha$ defines a measure on $X_{u_\alpha}$. This measure $\nu_\alpha$ is also $\sigma$-invariant,
and  the unique ergodicity of  
$(X_{u_\alpha}, \sigma)$ for each $\alpha\geq 0$ implies that $\nu_\alpha=\mu_\alpha$ for each $\alpha\geq 0$, which gives  $\nu=\mu$. 
\end{proof}

Next we extend results  of Dekking~\cite{Dekking}  concerning the discrete spectrum of constant-length substitution shifts. We refer to his article for all relevant definitions and to \cite{Downarowicz} for the definition of an odometer. The {\em maximal equicontinuous factor} of a topological dynamical system is  a rotation on a compact abelian group that is determined by  the collection of {\em continuous} eigenvalues of the system. An eigenvalue is continuous if there exists a continuous eigenfunction for that eigenvalue.  The set of continuous eigenvalues is generally a proper subset of the set of measurable eigenvalues. For primitive constant-length substitution shifts, the continuous and measurable eigenvalues coincide~\cite{Dekking}.

\begin{theorem}\label{discrete_spectrum}
Let $p$ be prime. 
Let $u$ be a fixed point of a profinite substitution $\theta=\varprojlim \theta_\alpha$, with each $\theta_\alpha$  length-$p$ and  primitive. Let $u_\alpha$ be a fixed point of $\theta_\alpha$, so that 
 $(X_u,  \mathcal B, \mu, \sigma)  =  \varprojlim \, (X_{u_\alpha},  \mathcal B_\alpha, \mu_\alpha, \sigma )$.
  Then the maximal equicontinuous factor of $(X_u,   \sigma )$ is an odometer, and every measurable eigenvalue is a continuous eigenvalue.
\end{theorem}

\begin{proof}
First suppose that there exists $\alpha$ such that  the  sequence $u_\alpha$ is not eventually periodic; then  for each $\beta\geq \alpha$,  $u_\beta$ is not eventually periodic. 
In this case Dekking's results~\cite{Dekking}  
tell us that  there exists some $\alpha_0$ such that for each $\alpha\geq \alpha_0$, $(\Z_p \times \Z/(h_\alpha \Z), +(1,1))$ is the maximal equicontinuous factor  of $(X_{u_\alpha},  \sigma)$, for some $h_\alpha$ coprime to $p$, and also that every measurable eigenvalue is a continuous eigenvalue. Note that  $h_\alpha \mid h_{\alpha+1}$ for each $\alpha \geq 0$. Without loss of generality, we suppose that $\alpha_0=1$.  Let $(\mathcal Z,+1)$ be the odometer formed by the sequence $ (h_{\alpha+1}/h_\alpha)_{\alpha\geq 0}$.
Then $(\Z_p \times \Z/(h_\alpha \Z), +(1,1))_{\alpha \geq 0}$ forms an inverse limit system in the category of group rotations,  and $\left(\Z_p \times \mathcal Z, +(1,1)\right)$ is its inverse limit.
Hence $\left(\Z_p \times \mathcal Z, +(1,1)\right)$ is an equicontinuous factor of the inverse limit system $(X_u,\sigma)$. If for each $\alpha \geq 0$, $u_\alpha$ is eventually periodic, with period $h_\alpha$, then the above argument follows through, except that in this case $(\mathcal Z,+1)$ is an equicontinuous factor of $(X_u,  \sigma)$.

Given a nontrivial measurable eigenfunction $f$ of   $(X_u,  \mathcal B, \mu, \sigma)$, with eigenvalue $\lambda$,
let $f_\alpha \colonequal \mathbb E(f \mid \mathcal B_{\alpha}^*)$ denote the conditional expectation of $f$ given $\mathcal B_{\alpha}^{*}$. Then $f_\alpha$ is isomorphic to  an eigenfunction of 
$(X_{u_\alpha}, \mathcal B_\alpha,  \sigma, \mu_\alpha)$
with eigenvalue $\lambda$, and $f_\alpha$ is nontrivial for large $\alpha$. For primitive substitutions, every measurable eigenvalue is a continuous eigenvalue. Also, $f_\alpha \rightarrow f$ $\mu$-almost everywhere, so that any measurable eigenvalue of $(X_u,  \mathcal B, \mu, \sigma)$, and hence any continuous eigenvalue, has already contributed to  the maximal equicontinuous factor  $\left(\Z_p \times \Z/(h_\alpha \Z),  +(1,1)\right)$  (or $(\Z/(h_\alpha \Z), +1)$) of $(X_{u_\alpha}, \sigma)$ for large enough $\alpha$.  This completes the proof that $\left(\Z_p \times \mathcal Z, +(1,1)\right)$ (or $(\mathcal Z,+1)$) is the maximal equicontinuous factor of $(X_u,   \sigma )$.
\end{proof}

\begin{example}\label{Fibonacci}
The Fibonacci sequence  $F=F(n)_{n\geq 0} = 0, 1, 1, 2, 3, 5, 8, 13, \dots$ is periodic when reduced modulo $m$ for any $m\geq 1$; hence for any prime $p$ and any $\alpha\geq 0$, $(F(n) \bmod p^\alpha)_{n\geq 0}$ is $p$-automatic.
Note that  because of the repeated $F(1) = F(2) = 1$, the letter-to-letter coding $\tau$ for which $F=\tau(u)$ is not the identity map.
For example, let $p = 2$, consider the alphabet  $\mathcal S = \{0, s, 1, 2, 3, 5, 8, 13, \dots\}$, where we are using integers as state names for convenience.
Let $\tau(s) = 1$ and $\tau(m) = m$ for all $m \in \N$.
Then the sequence $u = 0, s, 1, 2, 3, 5, 8, 13, \dots$ satisfies $F = \tau(u)$.
Moreover, the profinite substitution $\theta$ satisfies
\begin{align*}
	\theta(0) &= 0 \, s \\
	\theta(s) &= 1 \, 2 \\
	\theta(F(m)) &= F(2m) \, F(2m+1) \quad \text{for $m \geq 2$},
\end{align*}
and $u$ is a fixed point of $\theta$.
For each $\alpha$, the sequence $F \bmod 2^\alpha$ is the coding of a primitive substitution $\theta_\alpha$.
For example, the fixed point of the substitution $\theta_2 (0)= \theta_2(2)=0a$, $\theta_2(a)= \theta_2(3)= b 2$, $\theta_2 (b)=\theta_2(c)=3c$ and  projects  to $F\bmod{4}= 0,1,1,2,3,1\ldots$ via the coding $\tau_2(a)=\tau_2(b)=\tau_2(c)=1$. 

Let $\ell(m)$ denote the (minimal) period length of $(F(n) \bmod m)_{n \geq 0}$.
For prime $p$, Wall~\cite[Theorem~5]{Wall} showed that if $e$ is the smallest positive integer such that $\ell(p^e) \neq \ell(p)$, then $\ell(p^\alpha) = p^{\alpha+1-e} \ell(p)$ for $\alpha \geq e$. For $p=2$ we have $\ell(2)=3$ and $e=2$.
For each $\theta_\alpha$ then, $(X_{u_\alpha}, \sigma)$ is conjugate to the finite group $\left(\Z/(2^{\alpha-1}\Z) \times \Z/(3\Z), +(1,1)\right)$, and  so $\left(\Z_2 \times \Z/(3\Z), +(1,1)\right)$ is the maximal equicontinuous factor of $(X_u,\sigma)$. Furthermore, the maximal equicontinuous factor of $(X_a, \sigma)$, which must be contained in that of $(X_u,\sigma)$,  is also $\left(\Z_2 \times \Z /(3\Z), +(1,1)\right)$.
\end{example}

We end by showing that profinite substitutions are {\em recognizable}. Let $(X_u,\sigma)$ be the shift generated by a fixed point of  the profinite substitution $\theta$.
We say that $\theta$ is {\em recognizable} if for any 
$y\in X_u$ there is a unique way to write $y=\sigma^k(\theta (x))$ with $x \in X_u$ and $0\leq k< |\theta(x(0))|$.  Moss\'{e}~\cite{Mosse} showed that if $\theta$ is primitive and generates an aperiodic fixed point $u$, and $X_u$ is the {\em two-sided} shift generated by $\theta$, then $\theta$ is recognizable.  The two-sided shift can be thought of in two equivalent ways. It can be defined as the set of  bi-infinite sequences, all of whose subwords are a subword of $\theta^n(a)$ for some $n\geq 1$ and some letter $a$. It can also defined as the {\em natural extension} of the one-sided shift. In other words,  if $\tilde X_u$ is the one-sided shift, then the two-sided shift  $X_u$ is defined as 
\[ X_u\colonequal\{x=\varprojlim  x_\alpha: x_\alpha \in \tilde X_u \mbox{ and }   \sigma(x_{\alpha +1}) = x_\alpha  \mbox{ for each } \alpha\geq 0 \}.\]
If $(\tilde X_{u_\alpha},\sigma)_{\alpha \geq 0}$ is an inverse family of one-sided shifts,  then their natural extensions $(X_{u_\alpha},\sigma)_{\alpha \geq 0}$ form an inverse family of (two-sided) shifts.

\begin{theorem}\label{recognizable}
Let $p$ be prime, and 
let $\theta$ be a profinite substitution $\theta=\varprojlim \theta_\alpha$ that generates an inverse family of two-sided shifts $\{(X_{u_\alpha}, \sigma) : \alpha \geq 0\}$. Suppose that each $\theta_\alpha$ is length-$p$ and primitive. Then $\theta$ is recognizable.
\end{theorem}

\begin{proof}Let $x =\varprojlim x_\alpha\in \varprojlim X_{u_\alpha}$. 
Moss\'{e}'s theorem tells us that each $x_\alpha$ can be written in a unique way as $x_\alpha = \sigma^{k_\alpha} \theta_\alpha (y_\alpha)$ with $0\leq k_\alpha < p$.
Then
\[  \sigma^{k_\alpha}\theta_\alpha(y_\alpha)= x_\alpha = \psi_{\alpha,\beta}(x_\beta) =    \psi_{\alpha,\beta }  \sigma^{k_\beta}\theta_\beta(y_\beta) = \sigma^{k_\beta} \psi_{\alpha,\beta} \theta_\beta (y_\beta),  \]
where the final equality follows since  $\psi_{\alpha,\beta}$ is defined termwise.
Hence, since $\psi_{\alpha, \beta}$ commutes with $\theta_\beta$ for each $0 \leq \alpha \leq \beta$, we have $k_\beta = k_\alpha $ and  $\psi_{\alpha,\beta}(y_\beta) = y_\alpha$.
Let $k\colonequal k_\alpha$ and $y\colonequal\varprojlim y_\alpha$; then $x=\sigma^k \theta (y)$. If $x=\sigma^{k'} \theta (y')$, then for each $\alpha\geq 0$ we have $x_\alpha= \psi_{\alpha, \infty} (x) = \psi_{\alpha, \infty} (\sigma^{k'} \theta (y')     ) = \sigma^{k'} \theta_\alpha( \psi_{\alpha, \infty} ( y' ))$, so that uniqueness implies that $k=k'$ and   $y_\alpha= \psi_{\alpha, \infty} ( y' )$.
\end{proof}

\subsection{Limit sets of profinite substitutions}\label{limit sets}

Given a profinite substitution $\theta$, which defines a shift $(X_u, \sigma)$, we define the {\em limit set} $\mathcal L(\theta) \colonequal \bigcap_{n\geq 0} \theta^n (X_u)$.
The limit set is a nonempty compact set.
If $\theta$ is a substitution on a finite alphabet, then $\mathcal L(\theta)$  contains only  the periodic  points of $\theta$, of which there are finitely many.
However, when we consider substitutions on an infinite alphabet, $\mathcal L(\theta)$ consists of inverse limits of $\theta_\alpha$-periodic points, and the period lengths can increase. For example, let $a$ be the coding of a fixed point $u$ of a profinite substitution $\theta$. Define $\varphi :\Z_p^\N \rightarrow \Z_p$ by $\varphi(x(n)_{n\geq 0})= x(0)$. 
We have, for natural numbers $n$, $k$, and $r<p^n$,
\[
	a(kp^n+r) = \varphi \sigma^{kp^n+r}(a) =  \varphi \sigma^{kp^n+r}( \tau (u))=\varphi   \sigma^r \tau \sigma^{kp^n}(u) = \varphi \sigma^r \tau\!\left(\theta^n \sigma^k (u)\right) ,
\]
so that
\[
	a(kp^n+r)\in  \varphi \sigma^r \tau\!\left(\bigcap_{j\leq n}\theta^j (X_u)\right).
\]
Hence if $\lim_{n \to \infty} a(kp^n+r)$ exists, then 
 $\mathcal L(\theta)$ contains points other than $u$.
 For example, as mentioned in Section~\ref{Introduction}, the sequence $C(2^n)_{n \geq 0}$ converges in $\Z_2$, where $C(n)$ is the $n$th Catalan number.
More generally, we have the following.

\begin{proposition}[{\cite[Corollary~3.1]{michel_miller_rennie}}]\label{Catalan limits}
Let $p$ be prime, and let $C(n)$ be the $n$th Catalan number.
For each $k, r \in \Z$ with $k \geq 1$, the limit $\lim_{n \to \infty} C(kp^n +r)$ exists in $\Z_p$.
\end{proposition}

We have a similar result for the Fibonacci sequence.
More generally, we have such limits for any sequence satisfying a linear recurrence
\[
	a(n + \ell) + c_{\ell - 1} a(n + \ell - 1) + \dots + c_1 a(n + 1) + c_0 a(n) = 0
\]
with constant coefficients $c_i \in \Z_p$.
The \emph{characteristic polynomial} of this sequence is $x^\ell + \dots + c_1 x + c_0$.

\begin{proposition}[{\cite[Corollary~11]{Rowland--Yassawi_constant-recursive}}]\label{Fibonacci limits}
Let $p$ be prime, and let $a(n)_{n \geq 0}$ be a constant-recursive sequence of $p$-adic integers with monic characteristic polynomial $g(x) \in \Z_p[x]$.
There exists an integer $f \geq 1$ such that, for each $k, r \in \Z$ with $k \geq 1$, the limit $\lim_{n \to \infty} a(k p^{f n} + r)$ exists and is algebraic over $\Q_p$.
\end{proposition}

A suitable integer $f$ can be given explicitly as follows.
Let $K$ be a degree-$d$ splitting field of $g(x)$ over $\Q_p$ with ramification index $e$; then we can take $f = d/e$.
For example, for the Fibonacci sequence and $p = 2$ as pictured in Section~\ref{Introduction}, we obtain the value $f = 2$,
and the two limit points are $\pm \sqrt{-\frac{3}{5}}$ in $\Z_2$.

Unlike Proposition~\ref{Catalan limits}, the limit in Proposition~\ref{Fibonacci limits} comes from an \emph{approximate twisted interpolation} of the sequence $a(n)_{n \geq 0}$ to the relevant extension of $\Q_p$.
Amice and Fresnel~\cite{Amice--Fresnel} give an alternate characterization of sequences which have twisted interpolations.

\subsection{Cocycle sequences} \label{Cocycle sequences}

The examples we have worked with so far consist of sequences of $p$-adic integers whose generating function is algebraic over $\mathbb Z_p(x)$. In this section we show that certain {\em cocycle sequences}  are codings of the fixed point of a length-$p$ profinite substitution. Let $M_p$
 be the $p \times p$ matrix all of whose entries are $1$. This matrix is the incidence matrix for the substitution $\theta^*$ on the alphabet $\Z / (p\Z) $ defined as 
$\theta^*(j) = 01 \cdots (p-1)$ for each $j \in \Z / (p\Z)$, whose fixed point is periodic.

 We say that $\theta$ is {\em aperiodic} if it has a fixed point which is not 
 periodic. Let $\theta $ be any  aperiodic substitution on $\Z / (p\Z) $ whose incidence matrix is $M_p$.
 Let us assume also that $\theta(0)$ starts with 0, and let $u=0\cdots$ be the fixed point starting with 0.  Let  $\mu$ be the unique measure that is preserved by $\sigma$. Then the  shift 
 $(X_u, \sigma)$ has a {\em Bratteli--Vershik}
representation $(X_B,  \varphi_\theta)$~\cite{LV}, where $B$  is a  {\em Bratteli diagram} and $ \varphi_\theta$ is a {\em Vershik map:}
We briefly describe these objects.
The Bratteli diagram $B$ is an infinite directed graph, and for our example, 
we illustrate $B$ in Figure~\ref{Bratteli_diagram} for the case  $p=3$.

\begin{figure}[h]
\centerline{\includegraphics[scale=0.5]{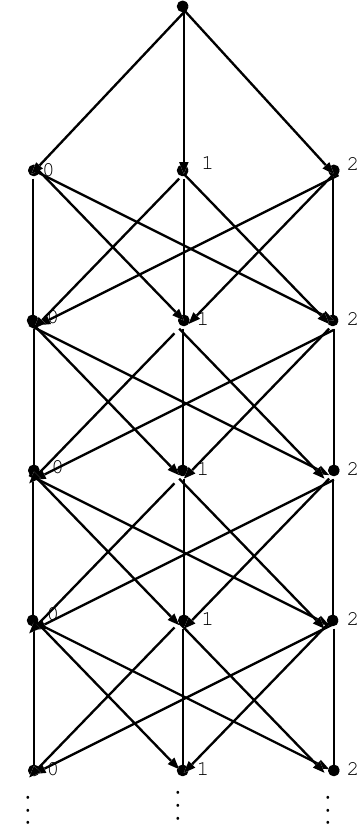}}
\caption{The Bratteli diagram associated with $M_3$.}
 \label{Bratteli_diagram}
\end{figure}

Apart from the root vertex at the top of the diagram, there are $p$ vertices at each level $n$, which we label $0, \ldots, p-1$, moving from left to right.
The levels are indexed by increasing indices $n$ as we move down in the diagram, $n=0, 1, \ldots$; we do not think of the root vertex as occupying a level. The edge structure for $B$ is determined by the matrix $M_p$.
Namely, the number of edges from vertex $i$ on level $n-1$ to vertex $j$ on level $n$ is the $(i,j)$ entry of $M_p$.
The substitution $\theta$  defines a linear order on the incoming edges to any vertex: if the vertex is labelled  $j$ and $\theta(j)\colonequal i_0 \cdots i_{p-1}$, then we give the edge with source $i_k$ the label $k$.
Let $X_B$ be the set of infinite paths in $B$ starting at the root vertex.
Such a path is labelled $x = x_0, x_1, \ldots $ where $x_n$ is the label of the edge from level $n$ to level $n+1$.
The linear order on the incoming edges to a vertex defines a partial order on  $X_B$. Namely, we can compare two infinite paths  $x$ and $x'$ in $B$ if and only if  they eventually agree: if $n$ is the smallest integer such that $x$ and $x'$ agree from level $n$ onwards, then we write  $x<x'$ if $x_{n-1} < x_{n-1}'$.
If $x$ is a non-maximal path (i.e.\ one of its edges is not maximally labelled) then it has a {\em successor} in this order. Namely, if $n$ is the smallest integer such that  $x_n$  is not a maximal edge, then the successor of $x$ agrees with $x$ from level $n+1$, is minimal up to level $n$, and the edge between level $n$ and level $n+1$ is the successor of $x_n$.

Thus any substitution $\theta$ with incidence matrix $M_p$ defines a partial ordering of 
$X_B$ and this determines a {\em Vershik map} $\varphi_\theta :X_B \rightarrow X_B$ where $\varphi_\theta(x)$ is defined to be the successor of $x$ in the ordering determined by $\theta$. Note that $\varphi$ is not defined on the set of maximal paths, but this is a finite set, and here we define it arbitrarily.

If $\theta$ is an aperiodic primitive substitution, then $(X_u, \mathcal B, \mu, \sigma)$  is  measurably conjugate to $(X_B, \overline{\mathcal B}, \overline { \mu}, \varphi_\theta )$  with $\overline 
{\mathcal B}$ the $\sigma$-algebra generated by cylinder sets on $X_B$ and $\overline \mu$ the image of $\mu$ via the conjugacy.
If $\theta$ is a periodic substitution (as $\theta^*$ is), then $(X_B,\varphi_\theta)$ is topologically conjugate to the $p$-adic odometer $(\Z_p, +1)$.  In this latter case,  if the finite path $x$ has edges labelled by the base-$p$ expansion of $m$, then $\varphi_{\theta^*}^n(x)$ is the finite path whose edges are labelled by the base-$p$ expansion of $m+n$.  
We refer the reader to \cite{LV} for details.

Note that the ordering induced by $\theta^*$ on $B$ has the special property that an edge labelled $i$ has as source a vertex labelled $i$.

Suppose that $\theta$ is aperiodic and $\theta(0)$ starts with $0$.  Let $0^\infty$ denote the minimal path in $B$ that runs through the vertices labelled $0$, and  let the sequence $s(n)_{n\geq 0}$ of natural numbers be defined by 
\[      \varphi_\theta^n(0^\infty) = \varphi_{\theta^*}^{s(n)}  (0^\infty)  .\]
The sequence $s(n)_{n\geq 0}$ is called a {\em cocycle}.

\begin{theorem}\label{cocycle}
Let $p$ be prime.
Let $M_p$ be the $p \times p$ matrix all of whose entries are $1$.
Let $\theta$ be an aperiodic substitution with incidence matrix $M_p$, and let $\theta^*(j) = 01 \cdots (p-1)$ for each $j \in \Z / (p\Z)$.
Suppose $\theta(0)$ starts with $0$.
Let $s(n)_{n \geq 0}$ be the cocycle defined by $\theta$ and $\theta^*$.
Then $(s(n) \bmod p^\alpha)_{n \geq 0}$ is the fixed point of a length-$p$ substitution $\theta_\alpha$ for every $\alpha \geq 0$.
\end{theorem}

\begin{proof}
Note that if the finite path $x_0 x_1 \cdots x_k$ passes through the vertices $v_0 v_1 \cdots v_{k+1}$, and if $x_0 x_1 \cdots x_k = \varphi_\theta^n (00 \cdots 0)  $, then 
$s(n) = \sum_{j=0}^{k+1}p^jv_j$, and $s(n) \bmod p^\alpha =  \sum_{j=0}^{\alpha-1}p^jv_j$.

Given $\alpha\geq 0$, we define a  substitution $\theta_\alpha$ on $ \Z/(p^\alpha \Z)$ of length-$p$ as follows. Given $j = j _0 p^0 +j_1 p^1 + \cdots + j_{\alpha - 1}p^{\alpha - 1} \in \Z/(p^\alpha \Z)$,
define $\theta_\alpha (j) = \theta(j_0) + p (  j _0 p^0 +j_1 p^1 + \cdots + j_{\alpha - 2}p^{\alpha - 2}    )$, where here we are adding $p (  j _0 p^0 +j_1 p^1 + \cdots + j_{\alpha - 2}p^{\alpha - 2} )$ to each entry in the word    $\theta(j_0)$. We claim that 
$(s(n) \bmod p^\alpha)_{n \geq 0}$ is a fixed point of $\theta_\alpha$.

To see this, we need to show that for each $n$, $\theta_\alpha (s(n) \bmod p^\alpha)= (s(pn), s(pn+1), \ldots, s(pn+p-1))\bmod p^\alpha$. To get  $s(pn+\ell) \bmod p^\alpha$, we need  the first $\alpha$ vertices through which the path $\varphi_{\theta}^{pn+\ell}(00\cdots 0)$ runs. Suppose that the path $\varphi_{\theta}^{n}(00\cdots 0)$ passes through the vertices $v_0, v_1, \ldots, v_{\alpha-1}$, so that 
$s(n) \bmod p^\alpha= \sum_{j=0}^{\alpha-1}v_j p^{j}$.
Recall that we use the notation $\theta (a)_j$ to denote the $j$-th letter of $\theta (a)$.
Then the  path $\varphi_{\theta}^{pn+\ell}(00\cdots 0)$ starts at the vertex labelled $\theta(v_0)_\ell$, followed by $v_0, \ldots, v_{\alpha-2}$ at levels $1, \ldots , \alpha-1$ of the diagram respectively. In other words, $s(pn+\ell) \bmod p^\alpha= \theta(v_0)_\ell + p \sum_{j=0}^{\alpha-2}v_jp^{j}
= \theta_\alpha (s(n) \bmod p^\alpha)_\ell$, as desired.
\end{proof}

 \begin{remark}
 Since cocycle sequences are bijections of $\N$, it is very easy to define the cocycle sequence as the fixed point of a length-$p$ substitution on $\N$. For example, if $p=2$ and $\phi(0)=01, $ $\phi(1)=10$ is the Thue--Morse substitution, then it has as transition matrix $M_2$ and its cocycle sequence
\[
	s(n)_{n \geq 0} = 0, 1, 3, 2, 7, 6, 4, 5, 15, 14, 12, 13, 8, 9, 11, 10, \dots
\]
is the fixed point of  the length-2 substitution $\theta$ on $\N$  defined by 
\[
\theta(m)=
	\begin{cases}
		(2m) \, (2m+1)	& \mbox{if $m$ is even}  \\
		(2m+1) \, (2m)	& \mbox{if $m$ is odd}.
	\end{cases}
\]
In particular, $s(n)_{n \geq 0}$ projects modulo $2$ to the Thue--Morse sequence.
However, it can be shown that $s(n)_{n\geq 0}$ is $2$-regular in the sense of Allouche and Shallit~\cite{Allouche--Shallit}; namely, we have the recurrence
\begin{align*}
	s(4 n) &= -2 s(n) + 3 s(2 n) \\
	s(4 n + 1) &= -2 s(n) + 2 s(2 n) + s(2 n + 1) \\
	s(4 n + 2) &= -2 s(n) + 3 s(2 n + 1) \\
	s(4 n + 3) &= -2 s(n) + s(2 n) + 2 s(2 n + 1).
\end{align*}
It follows that $(s(n) \bmod k)_{n \geq 0}$ is $2$-automatic for every $k \geq 2$~\cite[Corollary~2.4]{Allouche--Shallit}.
Therefore, by a theorem of Cobham, for a prime $p \neq 2$ the sequence $(s(n) \bmod p^\alpha)_{n \geq 0}$ is not $p$-automatic unless it is eventually periodic.
Moreover, the generating function $\sum_{n \geq 0} s(n) x^n$ is not rational,
so it follows from a result of B\'ezivin~\cite{Bezivin, Bell--Coons--Rowland} that $s(n)_{n \geq 0}$ is not algebraic, nor is it the diagonal of a rational function.

\end{remark}

\section*{Acknowledgments}
The authors acknowledge the hospitality and support of LaCIM, Montr\'{e}al and LIAFA, Universit\'{e} Paris 7.
We also thank the referee for excellent comments and suggestions.

{\footnotesize
\bibliographystyle{alpha}
\bibliography{bibliography}
}

\end{document}